\documentclass[12pt]{article}
\usepackage{latexsym,amssymb, amsthm, epsfig}
\usepackage{helvet}         
\usepackage{courier}        

\usepackage{color}
\usepackage{graphicx}
\usepackage{amsmath}
\usepackage{amsfonts}

\newtheorem{corollary}{\bf Corollary}[section]
\newtheorem{theorem}{\bf Theorem}[section]
\newtheorem{conjecture}{\bf Conjecture}[section]
\newtheorem{lemma}{\bf Lemma}[section]
\newtheorem{proposition}{\bf Proposition}[section]
\newtheorem{remark}{\bf Remark}[section]
\newtheorem{definition}{\bf Definition}[section]
\newcounter{for}[section]

\newcommand{\be}[1]{\addtocounter{for}{1}\begin{equation}\label{#1}}
\newcommand{\ee}{\end{equation}}

\def\E{{\mathbb E}}

\def\P{{\mathbb P}}
\def\R{{\mathbb R}}

\def\N{{\mathbb N}}

\def\C{{\mathcal{C}}}
\def\F{{\mathcal{F}}}

\def\ka6{6}

\def\({{\Bigl(}}
\def\){{\Bigr)}}

\def\one{{\mathbf 1}}

\def\square{\ifmmode\sqr\else{$\sqr$}\fi}
\def\sqr{\vcenter{
         \hrule height.1mm
         \hbox{\vrule width.1mm height2.2mm\kern2.18mm\vrule width.1mm}
         \hrule height.1mm}}                  

\theoremstyle{plain}

\theoremstyle{definition}

\theoremstyle{remark}

\def\epr{\end{proof}}

\def\bpr{\begin{proof}}

\def \beq {\begin{eqnarray}}
\def \eeq {\end{eqnarray}}
\def \beqn {\begin{eqnarray*}}
\def \eeqn {\end{eqnarray*}}

\newcommand{\bl}[1]{\begin{lemma}\label{#1}}
\newcommand{\br}[1]{\begin{remark}\label{#1}}
\newcommand{\brs}[1]{\begin{remarks}\label{#1}}
\newcommand{\bt}[1]{\begin{theorem}\label{#1}}
\newcommand{\bd}[1]{\begin{definition}\label{#1}}
\newcommand{\bp}[1]{\begin{proposition}\label{#1}}
\newcommand{\bc}[1]{\begin{corollary}\label{#1}}
\newcommand{\bfact}[1]{\begin{fact}\label{#1}}
\newcommand{\bex}[1]{\begin{example}\label{#1}}
\newcommand{\ec}{\end{corollary}}
\newcommand{\efact}{\end{fact}}
\newcommand{\eex}{\end{example}}
\newcommand{\el}{\end{lemma}}
\newcommand{\er}{\end{remark}}
\newcommand{\ers}{\end{remarks}}
\newcommand{\et}{\end{theorem}}
\newcommand{\ed}{\end{definition}}

\newcommand{\ep}{\end{proposition}}

\newcommand{\bcl}[1]{\begin{claim}\label{#1}}
\newcommand{\ecl}{\end{claim}}

\newcommand{\ecs}{\end{corollary}}
\newcommand{\eers}{\end{exercise}}
\newcommand{\eexs}{\end{example}}
\newcommand{\eems}{\end{example}}
\newcommand{\els}{\end{lemma}}
\newcommand{\eles}{\end{lemmaex}}
\newcommand{\ets}{\end{theorem}}
\newcommand{\eds}{\end{definition}}
\newcommand{\eps}{\end{proposition}}

\newcommand{\bi}{\begin{itemize}}
\newcommand{\ei}{\end{itemize}}
\newcommand{\ben}{\begin{enumerate}}
\newcommand{\een}{\end{enumerate}}

\def\vbar{\mathchoice{\vrule height6.3ptdepth-.5ptwidth.8pt\kern-.8pt}
   {\vrule height6.3ptdepth-.5ptwidth.8pt\kern-.8pt}
   {\vrule height4.1ptdepth-.35ptwidth.6pt\kern-.6pt}
   {\vrule height3.1ptdepth-.25ptwidth.5pt\kern-.5pt}}
\def\fudge{\mathchoice{}{}{\mkern.5mu}{\mkern.8mu}}
\def\bbc#1#2{{\rm \mkern#2mu\vbar\mkern-#2mu#1}}
\def\bbb#1{{\rm I\mkern-3.5mu #1}}
\def\bba#1#2{{\rm #1\mkern-#2mu\fudge #1}}
\def\bb#1{{\count4=`#1 \advance\count4by-64 \ifcase\count4\or\bba A{11.5}\or
   \bbb B\or\bbc C{5}\or\bbb D\or\bbb E\or\bbb F \or\bbc G{5}\or\bbb H\or
   \bbb I\or\bbc J{3}\or\bbb K\or\bbb L \or\bbb M\or\bbb N\or\bbc O{5} \or
   \bbb P\or\bbc Q{5}\or\bbb R\or\bbc S{4.2}\or\bba T{10.5}\or\bbc U{5}\or
   \bba V{12}\or\bba W{16.5}\or\bba X{11}\or\bba Y{11.7}\or\bba Z{7.5}\fi}}

\def \L {{\cal{L}}}

\def \C {{\cal{C}}}

\def\sqr#1#2{{\vcenter{\vbox{\hrule height .#2pt
                             \hbox{\vrule width .#2pt height#1pt \kern#1pt
                                   \vrule width .#2pt}
                             \hrule height .#2pt}}}}
\def\square{\mathchoice\sqr54\sqr54\sqr{4.1}3\sqr{3.5}3}
\def\pmb#1{\setbox0=\hbox{#1}%
   \kern-.025em\copy0\kern-\wd0
   \kern.05em\copy0\kern-\wd0
   \kern-.025em\raise.0433em\box0 }
\def\sqr#1#2{{\vcenter{\vbox{\hrule height.#2pt
     \hbox{\vrule width.#2pt height#1pt \kern#1pt
   \vrule width.#2pt}\hrule height.#2pt}}}}

\def\N{{\mathbb N}}

\def\R{{\mathbb R}}

\def\P{{\mathbb P}}

\def\d{{\rm d}}

\title{Front propagation and quasi-stationary distributions: the same selection principle?}

\author{Pablo Groisman\thanks{Departamento de Matem\'atica, Fac. Cs. Exactas y
Naturales, Universidad de Buenos Aires and IMAS-CONICET. {\tt
pgroisma@dm.uba.ar}, {\tt http://mate.dm.uba.ar/$\sim$pgroisma.}}
  \ and  Matthieu Jonckheere\thanks{IMAS-CONICET. {\tt mjonckhe@dm.uba.ar},
 {\tt http://matthieujonckheere.blogspot.com}.
} }

\date{}
\begin{document}
\maketitle
\abstract{We analyze the connection between selection principles in front
propagation and quasi-stationary distributions. We describe the missing link through the microscopic models known as Branching Brownian Motion with selection and Fleming-Viot.}

\bigskip



\section{Introduction}

A selection mechanism in front propagation can be thought of as follows: a certain phenomenology is described through an equation
that admits an infinite number of traveling-wave solutions, but there is only
one which has a physical meaning, the one with minimal velocity. Under
mild assumptions on initial conditions, the solution converges to this
minimal-velocity traveling wave.
 The most remarkable example of this fact is the celebrated F-KPP equation(for Fisher, Kolmogorov-Petrovskii-Piskunov)

\begin{equation}
 \label{KPP}
\begin{array}{l}
\displaystyle \frac{\partial v}{\partial t}= \displaystyle  \frac12 \frac{\partial^2 v}{\partial x^2} + r f(v), \quad x \in \R,\,
t>0,\\
\\
\displaystyle v(0,x)= \displaystyle v_0(x), \quad x \in \R.
\end{array}
\end{equation}

Assume for simplicity that $f$ has the form $f(s)=s^2 -s$, but this can be
generalized up to some extent. We also restrict ourselves to initial
data $v_0$ that are distribution functions of probability measures in $\R$. The equation was
introduced in 1937 \cite{Fisher, KPP} as a model for the evolution of a
genetic trait, and since then has been widely studied (in fact more than two thousand works refer to one of these papers).

Both Fisher and Kolmogorov, Petrovskii and Piskunov proved independently that
this equation admits an infinite number of traveling wave solutions of the form
$v(t,x) = w_c(x-ct)$ that travel at velocity $c$. This fact is somehow unexpected from the modeling point of view. In words of
Fisher \cite[p. 359]{Fisher}

\begin{quotation} ``Common sense would, I think, lead us to believe that, though
the velocity of advance might be temporarily enhanced by this method, yet
ultimately, the velocity of advance would adjust itself so as to be the
same irrespective of the initial conditions. If this is so, this equation
must omit some essential element of the problem, and it is indeed clear that
while a coefficient of diffusion may represent the biological
conditions adequately in places where large numbers of individuals of both
types are available, it cannot do so at the extreme front and back of
the advancing wave, where the numbers of the mutant and the parent
gene respectively are small, and where their distribution must be
largely sporadic''.
\end{quotation}

Fisher proposed a way to overcome this difficulty, related to the probabilistic
representation given later on by McKean \cite{MK},
weaving links between solutions to \eqref{KPP} and Branching Brownian Motion.
The general
principle behind is that microscopic effects should be taken into account
to properly describe the physical phenomena. With a similar point of view in
mind, Brunet, Derrida and coauthors
\cite{BDMM,BDMM2,BD2,BD}
started in the nineties a study of the effect of microscopic noise in front
propagation for equation \eqref{KPP} and related models, which resulted in a
huge number of works that study the change in the behavior of the front when
microscopic effects are taken into account. These works include both
numerical and heuristic arguments \cite{BDMM, BDMM2,BD2, BD,  KL} as well as
rigorous proofs \cite{BG,BBS,DR,M2,M}. Before that, Bramson et.al
\cite{Betal} gave the first rigorous proof of a microscopic model for
\eqref{KPP}
that has a unique velocity for every initial condition. They also prove that
these velocities converge in the macroscopic scale to the minimum velocity of
\eqref{KPP}, and call this fact a {\em microscopic selection principle},
as opposed to the macroscopic selection principle stated above, that holds for
solutions of the hydrodynamic equation.

The theory of quasi-stationary distributions (QSD) has their own counterpart.
It is a typical situation that there is an infinite number of quasi-stationary
distributions, but the {\em Yaglom limit} (the limit of the conditioned
evolution of the process started from a deterministic initial condition)
  selects the minimal one, i.e. the one
with minimal expected time of absorption.

Up to our knowledge, despite of a shared feeling that similar principles do
occur in the context of QSD and of traveling waves, this relation has
never been stated precisely. The purpose of this note is to show that they are
two faces of the same coin. We first explain this link
through the example of Brownian motion. Then we show how to extend these
results to more general L\'evy processes.

The paper is organized as follows. In Section \ref{sec:macro},
we introduce traveling waves and QSDs as macroscopic models.
We focus in particular on the KPP equation and the links between its traveling waves and the QSDs of a drifting Brownian motion.
In Section \ref{sec:micro}, we introduce particle systems enlightening the selection principles
observed for the macroscopic models.
Finally in Section \ref{sec:Levy} we export these observations to more general models. We consider
general L\'evy processes under suitable assumptions and analyze selection principles in this context.

\section{Macroscopic models}\label{sec:macro}

We elaborate on the two macroscopic models we study: front propagation
and QSD.

\subsection{Front propagation in the KPP}

Since the seminal papers \cite{Fisher, KPP}, equation
\eqref{KPP} has received a huge amount of attention for several reasons. Among them,
it is one of the simplest models explaining several phenomena that are
expected to be universal. For instance, it admits a continuum of traveling wave solutions that can be parametrized by
their velocity $c$. More precisely, for each $c \in [\sqrt{2r}, +\infty)$ there
exists a function $w_c \colon \R \to [0,1]$ such that
\[
 v(t,x)=w_c(x-ct)
\]
is a solution to \eqref{KPP}. For $c< \sqrt{2r}$, there is no traveling wave solution,  \cite{AW,KPP}. Hence $c^*=\sqrt{2r}$
represents the
minimal velocity and $w_{c^*}$ the minimal traveling wave. Moreover,
if $v_0$ verifies for some $0<b < \sqrt{2r}$
\[
 \lim_{x\to \infty} e^{bx}(1-v_0(x)) = a >0,
\]
then
\begin{equation}
\label{domain.tw}
 \lim_{t \to  \infty} v(t,x+ct)=w_c(x), \quad \mbox{for }c=r/b + \frac12 b,
\end{equation}
see \cite{MK,MK2}. If the initial measure has compact
support (or fast enough decay at infinity), the solution converges to the
minimal traveling wave and the domain of attraction and velocity of
each traveling wave is determined by the tail of the initial distribution
\cite{AW,MK,U}. A smooth traveling wave solution of \eqref{KPP} that travels at
velocity $c$ is a solution to
\begin{equation}
 \label{TWKPP}
\frac12 w'' + cw + r(w^2 -w)= 0.
\end{equation}

The
behavior at infinity of these traveling waves is given by
\[
 1- w_c(x) \sim \begin{cases}
		    c_1e^{-bx} & c> \sqrt{2r}\\
		    c_2xe^{-x\sqrt{2r}} & c= \sqrt{2r}.
                \end{cases}
\]
This behavior is determined by the linearization of \eqref{TWKPP} at $w=1$,
i.e. the solution of
\begin{equation}
 \label{TWKPPLin}
\frac12 w'' + cw' + rw = 0.
\end{equation}
See \cite{SH, U}. We come back to this equation when dealing with
quasi-stationary distributions.

\subsection{Quasi-stationary distributions}
\label{qsd}
Quasi-stationary distributions have been extensively studied since the
pioneering work of Kolmogorov (1938), Yaglom (1947) and Sevastyanov (1951)
on the behavior of Galton-Watson processes.

The beginning of this theory and an important part of the research in the area has been motivated by
models on genetics and population biology, where the notion of
quasi-stationarity is completely natural to describe the behavior of populations
that are expected to get extinct,
conditioned on the event that extinction has not yet occurred, on large time
scales.

Being more precise, consider a Markov process $Z=(Z_t,\, t\ge 0)$, killed
at some state or region that we call $0$. The absorption time is defined
by $\tau=\inf\{t>0 \colon Z_t
\in 0\}$. The conditioned evolution at time $t$ is defined by
\[
 \mu_t^\gamma(\cdot):= \P_\gamma(Z_t\in  \cdot|\tau>t).
\]

Here $\gamma$ denotes the initial distribution of the process. A probability measure
$\nu$ is
said to be a quasi-stationary distribution (QSD) if $\mu_t^\nu = \nu$ for all
$t\ge 0$.

For Markov chains in finite state spaces, the existence and uniqueness of QSDs as well as the convergence of the conditioned evolution to this unique QSD
for every initial measure follows from Perron-Frobenius theory.
The situation is more delicate for unbounded spaces as there can be $0$, $1$ or an infinite number of QSD.
Among those distributions, the {\em minimal} QSD is the one that minimizes
$\E_\nu(\tau)$.

The {\em Yaglom limit} is a probability measure $\nu$ defined by
\[
\nu:= \lim_{t \to \infty} \mu_t^{\delta_x},
\]
if it exists and does not depend on $x$. It is known that if the Yaglom limit
exists, then it is a QSD. A general principle is that the Yaglom limit
{\em selects} the minimal QSD, i.e. the Yaglom limit is the QSD with minimal mean absorption time. This fact has been proved for a
wide class of processes that include birth and death process,
Galton-Watson
processes, random walks, Brownian motion, more general L\'evy processes, etc.

It can also be proved for $R$-positive processes by means of the theory of
$R$-positive matrices \cite{SVJ}.
To give a flavor of the results that hold in this situation, consider
a discrete time Markov chain in $\N$ that it is absorbed at $0$. Denote
$p=(p(i,j),\,\, i, j \in \N)$ its transition matrix so that $p$ is
sub-stochastic. We use $p^{(n)}$ for the $n$-th power of $p$. We say that $p$ is $R$-positive if one (and hence both) of the
following equivalent statements hold
\begin{enumerate}
 \item For some $i$ and $j$, the sequence $R^np^{(n)}(i,j)$ tends to a finite
non-zero limit as $n\to \infty$.
\item There exist non-negative, non-zero eigenvectors
$\nu=(\nu(k))_{k\in\N}, \, \beta=(\beta(k))_{k\in\N}$ associated to the
eigenvalue $1/R$ such that $\sum_{k=1}^\infty \nu(k) \beta(k) < \infty$.
\end{enumerate}
%
In 1966, Seneta and Vere-Jones proved the following theorem
\begin{theorem}[Seneta and Vere-Jones, \cite{SVJ}] Assume that the matrix $p$
is $R$-positive, then the conditioned evolution converges to $\nu$ as $n\to
\infty$ if one of the following conditions hold
\begin{enumerate}
 \item The left eigenvector $\nu$ satisfies $\sum \nu(i) <\infty$
and the initial distribution $\mu$ is dominated (pointwise) by a multiple of
$\nu$.
\item The right eigenvector $\beta$ is bounded away from zero and $\sum
\mu(j)\beta(j)<\infty$
\end{enumerate}
 \label{SenetaVere-Jones}
\end{theorem}

Observe that in both situations we have that $\nu$ is the minimal QSD and the
Yaglom limit. Also every initial distribution with tail light enough is in the
domain of attraction of $\nu$.
So, in the $R$-positive case, the situation is pretty clear. These results can be applied
for instance to the Galton-Watson process. In that case, a detailed study of the domain of
attraction of the other QSDs (which are parametrized by an interval) is given in
\cite{RVJ} where it can be seen that the limiting conditional distribution is
given by the tail of the initial distribution.

Unfortunately, on the one hand $R$-positivity is a property difficult to check and on the
other hand, there is a lot of interesting processes that are not $R$-positive.
For example, a birth and death process with constant drift towards the
origin, has a continuum of QSDs and initial distributions with
light tails are attracted by the minimal QSD, which can be computed explicitly
\cite{FMP}, but this process is not $R$-positive and hence Theorem
\ref{SenetaVere-Jones} does not apply.

%

\

The presence of an infinite number of quasi-stationary distributions is
something anomalous from the modeling point of view, in the sense that
no physical nor biological meaning has been attributed to them. The
reason for their presence here and in the front
propagation context is similar: when studying for instance population or
genes dynamics through the conditioned evolution of a Markov process, we are
implicitly considering an infinite population and microscopic effects are lost.

\medskip

So, as Fisher suggests, in order to avoid the undesirable infinite number
of QSD, we should take into account microscopic effects. A natural way to
do this is by means of interacting particle systems. We discuss this in Section \ref{particle.systems}.

%

\paragraph{Brownian Motion with drift.} Quasi-stationary ditributions for
Brownian Motion with constant drift towards the origin are studied
in \cite{MPSM, MSM}. We briefly review here some of the results of these
papers and refer to them for the details.

For $c>0$ we consider a one-dimensional Brownian Motion $X=(X_t)_{t\ge0}$
with drift $-c$ defined by $X_t=B_t - ct$. Here $B_t$ is a one dimensional
Wiener process defined in the standard Wiener space. We use $\P_x$ for the
probability defined in this space such that $B_t$ is Brownian Motion started at
$x$ and $\E_x$ for expectation respect to $\P_x$. Define the
hitting time of zero, when the process is started at $x>0$ by $\tau_x(c)=\inf \{
t >0 \colon X_t=0\}$ and denote with $P_t^c$ the submarkovian semigroup defined
by
\begin{equation}
\label{qsd.semigroup}
 P_t^cf(x)=\E_x(f(X_t) \one_{\{\tau_x(c)>t\} }).
\end{equation}

In this case, differentiating \eqref{qsd.semigroup} and after some
manipulation it can be seen that the conditioned evolution $\mu P_t$ has a
density $u(t,\cdot)$ for every $t>0$ and verifies
\begin{equation}
 \label{cond.ev.BM}
\begin{array}{rcl}
\displaystyle \frac{\partial u}{\partial t}(t,x) &=  & \displaystyle\frac{1}{2}
\frac{\partial^2 u}{\partial^2 x}(t,x) + c \frac{\partial u}{\partial
x}(t,x)  + \frac12\frac{\partial u}{\partial
x}(t,0)u(t,x), \quad t>0, x>0,\\
u(t,0)&=& u(t,+\infty)=0,  \quad t>0,\\
\end{array}
\end{equation}

Recall now that a probability measure $\nu$ in $\R_+$ is a QSD if
\[
 \P_\nu(X_t \in \cdot|X_t>0) = \nu(\cdot).
\]
It is easy to check that if $\nu$ is a QSD, the hitting time of zero,
started with $\nu$ is an exponential variable of parameter $r$ and hence $\nu$
is a QSD if and only if there exists $r>0$ such that
\[
 \nu P_t^c = e^{-r t}\nu, \quad \mbox{for any} \quad t>0.
\]

Differentiating \eqref{qsd.semigroup} and using
the semigroup property we get that $\nu$ is a QSD if and only if
\begin{equation}
 \int (\frac12 f'' - cf')\,d\nu = -r \int f\, d\nu, \quad \mbox{for all } f \in
C_0^\infty(\R_+).
\end{equation}
Integrating by parts we get that the density $w$ of $\nu$ must verify
\begin{equation}
\label{qsd.brownian}
 \frac12 w'' + cw' + r w = 0.
\end{equation}
Solutions to this equation with initial
condition $w(0)=0$ are given by
\[
w(x)= \begin{cases}
      me^{-c x} \sin(\sqrt{c^2- 2r}x)& r > \frac{c^2}2,\\
      mx e^{- c x} & r = \frac{c^2}2,\\
      me^{-c x} \sinh(\sqrt{c^2- 2r}x)& r < \frac{c^2}2.\\
 \end{cases}
\]

Observe that $w$ defines an integrable
density function if and only if $0<r\le c^2/2$ (or equivalently, $c\ge
\sqrt{2r}$). One can thus parametrize the set of QSDs by their eigenvalues $r$,
$\{\nu_r
\colon 0<r\le c^2/2\}$. For each $r$, the distribution function of $\nu_r$,
$v(x)=\int_0^x w(y)\,dy$ is a monotone solution of \eqref{qsd.brownian} with
boundary conditions
\begin{equation}
 \label{qsd.brownian.bc}
v(0)=0,\qquad v(+\infty)=1,
\end{equation}
which is the same equation \eqref{TWKPPLin} but in a different domain. The
following theorem characterizes the domain of attraction of each QSD.

\begin{theorem}[Mart\'inez, Picco, San Mart\'in, \cite{MPSM}] Let $\gamma$
be a probability measure on $(0,+\infty)$ with density $\rho$. If
\begin{equation}
\label{tail.qsd}
 \lim_{x \nearrow \infty} -\frac{1}{x} \log \rho(x) = b < c,
\end{equation}
then $\lim_t \mu^\gamma_t = \nu_{r(b)}$,
where $r(b)=cb - b^2/2$.
\end{theorem}
Observe that this last equation is equivalent to $c=r/b +
\frac12 b$, which should be compared with \eqref{domain.tw}. Also remark that if \eqref{tail.qsd} holds then
\[
 \lim_{x\nearrow \infty} -\frac{1}{x} \log \mu([x,+\infty) = b.
\]
We come back to equation \eqref{qsd.brownian} later, shedding light on the
links between QSD and traveling waves.

Finally let us mention that a similar result holds for the discrete-space
analog, i.e. birth and death processes with constant drift towards the origin
\cite{FMP2,FMP}

\section{Particle systems}\label{sec:micro}
\label{particle.systems}
In this section we introduce two particle systems. The first one is known as
Branching Brownian Motion (BBM) with selection of the $N$ right-most particles
($N-$BBM). As a consequence of the link between BBM and F-KPP that we
describe below, this process can be thought of as a microscopic version of F-KPP.
The second one is called Fleming-Viot and was introduced by Burdzy, Ingemar,
Holyst and March \cite{BHIM}, in the context of Brownian Motion in a
$d$-dimensional bounded domain. It is a slight variation of the original
one introduced by Fleming and Viot \cite{FV}. The first interpretation of this
process as a microscopic version of a conditioned evolution is due
to Ferrari and Maric \cite{FM}.

\subsection{BBM and F-KPP equation}

One-dimensional supercritical Branching Brownian Motion is a well-understood object. Particles diffuse following standard Brownian Motion started at
the origin and branch at rate 1 according to an offspring distribution that we
assume for simplicity to be $\delta_2$. When a particle branches, it has two
children and then dies. As already underlined, its connection with the F-KPP
equation and traveling
waves was pointed out by McKean in the seminal paper \cite{MK}. Denote with
$N_t$ the number of particles alive at time $t\ge 0$ and $\xi_t(1) \le  \dots
\le \xi_t(N_t)$ the position of the particles enumerated from left to right.
McKean's representation formula states that if $0\le v_0(x) \le 1$ and we
start the process with one particle at $0$ (i.e. $N(0)=1$, $\xi_0(1)=0$), then
\[
 v(t,x):= \E\left(\prod_{i=1}^{N_t}v_0(\xi_t(i)+x)\right)
\]
is the solution of \eqref{KPP}. Of special interest is the case where the
initial condition is the Heaviside function $v_0 = \one\{[0,+\infty)\}$ since in
this case
\[
 v(t,x)=\P(\xi_t(1)+x>0)=\P( \xi_t(N_t) <x).
\]
This identity as well as various martingales obtained as functionals of
this process
have been widely exploited to obtain the precise behavior of solutions of
\eqref{KPP}, using
analytic as well as probabilistic tools \cite{B,B2,HHK,SH,MK,U}.

\subsection{$N-$BBM and Durrett-Remenik equation}

Consider now a variant of BBM where the $N$ right-most particles are selected.
In other words, each time a particle branches, the left-most one is killed,
keeping the total number of particles constant.

This process was introduced by Brunet and Derrida \cite{BD2,BD} as part of a
family of models of branching-selection particle systems to study the effect of
microscopic noise in front propagation. By means of numerical simulations and
heuristic arguments, they conjectured that the linear speed of $N-$BBM differs
from the speed of standard BBM by  $(\log N)^{-2}$ and in a series of papers
with coauthors they study various statistics of the process
\cite{BDMM,BDMM2,BD2,BD}.
B\'erard and Gou\'er\'e proved this  shift in the velocity for a similar
process. We refer to the work of Maillard \cite{M2} for the most
detailed study of this process.

Durret and Remenik \cite{DR} considered a slightly different process in the
class of Brunet and Derrida: $N-$BRW. The system starts with $N$ particles.
Each particle gives rise to a child at rate one. The position of the
child of a particle at $x\in\R$ is $x+y$, where $y$ is chosen according to a
probability distribution with density $\rho$, which is assumed
symmetric and with finite expectation. After each birth, the $N+1$ particles are
sorted and
the left-most one is deleted, in order to keep always $N$ particles.
They prove that
 the empirical measure of this system converges to a deterministic probability measure
$\nu_t$ for every $t$, which is absolutely continuous with density
$u(t,\cdot)$, a solution of the following free-boundary problem
\begin{equation}
 \label{Durrett-Remenik}
\begin{array}{rcl}
\mbox{{\em Find }}(\gamma,u) \mbox{{\em such that}}& \\
\\
\displaystyle \frac{\partial u}{\partial t}(t,x) &  =  &
\displaystyle \int_{-\infty}^\infty u(t,y)\rho(x - y) \, dy \quad \forall x >
\gamma(t),\\
\displaystyle \int_{\gamma(t)} ^\infty u(t,y) \, dy  & = &1, \quad u(t,x)=0,
\quad \forall x \le \gamma(t),\\
u(0,x) & = &u_0(x).\\
\end{array}
\end{equation}
They also find all the traveling wave solutions for this equation. Just as
for the BBM, there exists
a minimal velocity $c^*\in\R$ such that for $c\ge c^*$ there is a unique traveling wave
solution with speed $c$ and no traveling wave solution with speed $c$
for $c < c^*$. The value $c^*$ and the behavior
at infinity of the traveling waves can be computed explicitly in terms of the
Laplace transform of the random walk. In Section \ref{sec:Levy} we show that these traveling waves correspond to QSDs of drifted random walks.

It follows from renewal arguments that for each $N$, the
process seen from the
left-most particle is ergodic, which in turn implies the existence of a velocity
$v_N$ at which the empirical measure travels for each $N$. Durrett and Remenik
prove that these velocities are increasing and converge to $c^*$ as $N$ goes to
infinity.

We can interpret this fact as a {\em weak selection principle}: the
microscopic system has a unique velocity for each $N$
(as opposed to the limiting equation) and the velocities converge to the minimal
velocity of the macroscopic equation. The word ``weak'' here refers to the fact
that only convergence of the velocities is proved, but not convergence of the
empirical measures in equilibrium.

In view of these results, the same theorem is expected to
hold for a $N-$BBM that branches at rate $r$. In this case the limiting equation
is conjectured to be given by
\begin{equation}
\label{Durrett-Remenik.BM}
\begin{array}{rcl}
\mbox{{\em Find }}(\gamma,u) \mbox{{\em  such that}}& \\
\\
\displaystyle \frac{\partial u}{\partial t}(t,x) &  =  &
\displaystyle \frac12\frac{\partial^2u}{\partial^2x}(t,x) + r u(t,x) \quad
\forall x >
\gamma(t),\\
\displaystyle \int_{\gamma(t)} ^\infty u(t,y) \, dy  & = &1, \quad u(t,x)=0,
\quad \forall x \le \gamma(t),\\
u(0,x) & = &u_0(x).\\
\end{array}
\end{equation}

The empirical measures in equilibrium are also expected to converge
to the minimal traveling wave. More precisely,
\begin{conjecture}
Both N-BBM and N-BRW are ergodic, with (unique) invariant
measure $\lambda^N$ and the empirical measure
distributed according to $\lambda^N$ converges to the delta measure supported on the minimal quasi-stationary distribution.
\end{conjecture}


\paragraph{Traveling waves.}
Let us look at the traveling wave solutions
$u(t,x)=w(x-ct)$ of \eqref{Durrett-Remenik.BM}. Plugging-in in
\eqref{Durrett-Remenik.BM} we see that they must verify
\begin{equation}
\label{TWDR}
\frac12 w'' + cw' + rw = 0, \quad w(0)=0, \quad \int_0^\infty w(y)\, dy=1.
\end{equation}
Which is exactly $\eqref{qsd.brownian}$. Note nevertheless that in \eqref{TWDR}
the parameter $r$ is part of the data of the problem (the branching rate) and
$c$ is part of the unknown (the velocity), while in
\eqref{qsd.brownian} the situation is reversed: $c$ is data (the drift)
and $r$ unknown (the absorption rate under the QSD). However, we have the
following relation

\[
 \begin{array}{ccc}
   \begin{array}{c}
    c \mbox{ is a minimal velocity for } r\\
    \mbox{in \eqref{TWDR}}
   \end{array} &
    \iff
   \begin{array}{c}
    r \mbox{ is a maximal absorption rate for } c\\
    \mbox{in \eqref{qsd.brownian}}
   \end{array}
 \end{array}
\]

Observe also that $1/r$ is the mean absorption time for the QSD associated to
$r$ and hence, if $r$ is maximal, the associated QSD is minimal. So the minimal
QSD for Brownian Motion in $\R_+$ and the minimal velocity traveling wave of
\eqref{Durrett-Remenik.BM} are one and the same. They are given by
\[
 u_{c^*(r)}(x)=u_{r^*(c)}= 2r^*xe^{-\sqrt{2r^*}} = (c^*)^2 x e^{- c^* x},
\]
which is the one with fastest decay at infinity.

Again, the distribution function $v$ of $u$ is a monotone solution to the same
problem but with boundary conditions given by $v(0)=0$, $v(+\infty)=1$.

It is worth noting that although the solutions to
\eqref{qsd.brownian} and \eqref{TWKPPLin} are not the same since they are
defined in different domains, there is a natural way to identify them (and also
with solutions of \eqref{TWKPP}). Given positive
constants $c$ and $r$, there is a solution $w$ of \eqref{TWKPP} if and only if
there is solution $\tilde w$ of \eqref{qsd.brownian}. Moreover, we have
\[
 \lim_{x\to \infty} \frac{w(x)}{\tilde w(x)}=1.
\]
In this case, the proof of this statement is immediate since solutions of
\eqref{qsd.brownian} and \eqref{TWKPPLin} are explicit and the relation among
solutions of \eqref{TWKPPLin} and \eqref{TWKPP} is very well understood
\cite{SH}. We conjecture that the same situation holds in much more
generality.

%
%
%
%

\subsection{Fleming-Viot and QSD}

The Fleming-Viot process can be thought of as a microscopic version of
conditioned evolutions. Its dynamics are built with a continuous time Markov
process $Z=(Z_t, \, t\ge0)$ taking values in
the metric space $\Lambda \cup \{0\}$, that we call the {\em driving
process}. We assume that $0$ is absorbing in the sense that
\[
 \P(Z_t =0|Z_0 = 0) = 1, \qquad \forall t\ge 0.
\]
We use $\tau$ for the absorption time
\[
 \tau=\inf\{t>0 \colon Z_t \notin \Lambda \}.
\]
As, before, we use $P_t$ for the submarkovian semigroup defined by
\[
P_t f(x) = \E_x(f(Z_t)\one\{\tau>t\}).
\]
For a given $N\ge 2$, the Fleming-Viot process is an interacting particle
system with $N$ particles. We use $\xi_t=(\xi_t(1),\dots,\xi_t(N)) \in
\Lambda^N$ to denote the state of the process, $\xi_t(i)$ denotes the position
of particle $i$ at time $t$. Each particle evolves according to $Z$ and
independently of the others unless it hits $0$, at which time,
it chooses one
of the $N-1$ particles in $\Lambda$ uniformly and takes its position.
The guenine definition of this process is not obvious and in fact is not true in
general. It can be easily constructed for processes with bounded jumps to
$0$, but is much more delicate for diffusions in bounded domains
\cite{BBF,GK} and it does not hold for diffusions with a strong drift close to
the boundary of $\Lambda$, \cite{BBP}.

Here we are also interested in the empirical measure of the process
\begin{equation}
\label{empirical.measure}
 \mu^N_t = \frac1 N \sum_{i=1}^N \delta_{\xi_t(i)}.
\end{equation}
Its evolution is mimicking the conditioned
evolution: the mass lost from $\Lambda$, is redistributed in $\Lambda$
proportionally to the mass at each state. Hence, as $N$ goes to infinity, we
expect to have a deterministic limit given by the conditioned evolution of the
driving process $Z$, i.e.
\[
 \mu^N_t(A) \to \P(Z_t \in A | \tau >t) \qquad (N\to\infty).
\]
This is proved in $\cite{V}$ by the Martingale method in great generality. See
also \cite{GJ} for a proof based on sub and super-solutions and correlations
inequalities.
A much more subtle question is the ergodicity of the process for fixed $N$
and the
behavior of these invariant measures as $N\to \infty$.
As a general principle it is expected that
\begin{conjecture}
If the driving process $Z$ has a Yaglom limit $\nu$, then
the Fleming-Viot process driven by $Z$ is ergodic, with (unique) invariant
measure $\lambda^N$ and the empirical measures \eqref{empirical.measure}
distributed according to $\lambda^N$ converge to $\nu$.
\end{conjecture}
We refer to \cite{GJ} for an extended discussion on this issue.
This conjecture has been proved
for subcritical Galton-Watson processes, where a continuum of QSDs arises
\cite{AFGJ}.

We have again here a {\em microscopic selection principle}: whereas
there exists an infinite number of QSDs, when microscopic effects are taken
into account (through the dynamics of the Fleming-Viot process), there is a
unique stationary distribution for the empirical measure, which selects
asymptotically the minimal QSD of the macroscopic model.

When the driving process is a one dimensional Brownian motion with drift
$-c$ towards the origin as in Section \ref{qsd}, the proof of the whole
picture remains open, but the ergodicity of FV for fixed $N$ has been recently
proved \cite{AT,AK}.

So, from \cite[ Theorem 2.1]{V} we have that for every $t>0$, $\mu^N_t$
converges as $N\to \infty$ to a measure $\mu_t$ with density $u(t,\cdot)$ satisfying (\ref{cond.ev.BM}).
 The open problem is to prove a similar statement in equilibrium. Observe
that $u$ is a stationary solution of \eqref{cond.ev.BM} if and only if it solves
\eqref{qsd.brownian} for some $r>0$. Hence, although equations
\eqref{Durrett-Remenik.BM} and
\eqref{cond.ev.BM} are pretty different, stationary solutions to
\eqref{cond.ev.BM} coincide with traveling waves of \eqref{Durrett-Remenik.BM}.

\subsection{Summing up}

\begin{enumerate}

\item The link between $N-$BBM and Fleming Viot, in the Brownian Motion case
is clear.
Both processes evolve according to $N$ independent Brownian
Motions and branch into two particles. At branching times, the left-most
particle is eliminated (selection) to keep the population size constant. The
difference is that while $N-$BBM branches at a constant rate $N r$,
Fleming-Viot branches each time a particle hits 0. This explains why in the
limiting equation for $N-$BBM the branching rate is data and the velocity is
determined by the system while in the hydrodynamic equation for Fleming-Viot the
velocity is data and the branching rate is determined by the system.

\item The empirical measure of $N-$BBM is expected to converge in finite time intervals to the solution of \eqref{Durrett-Remenik.BM}. This is supported by the results of \cite{DR} where the same result is proved for random walks.

\item The empirical measure of Fleming-Viot driven by Brownian Motion converges in finite time intervals
to the solution of \eqref{cond.ev.BM}.

\item Both $N$-BBM seen from the left-most particle and FV are ergodic and their empirical measure in equilibrium is
expected to converge to the deterministic measure given by the minimal solution
of \eqref{TWDR}.

Note though that while for $N-$BBM $r$ is data and minimality refers to $c$, for
Fleming-Viot $c$ is data and minimality refers to $1/r$ (microscopic selection
principle).

\item $u(t,x)=w(x-ct)$ is a traveling wave solution of
\eqref{Durrett-Remenik.BM} if and only if $w$ is the density of a
QSD for Brownian Motion with drift $-c$ and eigenvalue $-r$.

\item $c$ is minimal for $r$ (in \eqref{TWDR}) if and only if $1/r$ is
minimal for $c$. So, we can talk of a ``minimal solution of
\eqref{TWDR}'', which is both a minimal QSD and a minimal velocity traveling
wave.

\item The microscopic selection principle is conjectured to hold in both
cases,
with the same limit, but a proof is still unavailable.

\end{enumerate}


\section{Traveling waves and QSD for L\'evy processes}\label{sec:Levy}

Let $Z=(Z_t, \, t\ge 0)$ be a L\'evy process with values in $\R$, defined on
a filtered space $(\Omega, \F, (\F_t),\P)$ and Laplace exponent
$\psi: \mathbb R \to  \mathbb R $ defined by
\[
\E(e^{\theta Z_t})= e^{\psi(\theta ) t}
\]
such that
\[
\psi(\theta)= b\theta+ \sigma^2 \frac{\theta^2}{2} + g(\theta),
\]
where $b\in\R$, $\sigma>0$ (which ensures that $Z$ is non-lattice) and $g$ is defined in terms of the jump measure $\Pi$ supported in $\R\setminus\{0\}$ by
\[
g(\theta)=\int_x (e^{\theta x}- 1 -\theta x \one_{\{|x|<1\}}) \Pi(dx),  \qquad \int_\R (1 \wedge x^2) \Pi(dx) < \infty.
\]
Let $\theta^\star=\sup\{\theta\colon |\psi(\theta)|<\infty\}$ and recall that $\psi$ is strictly convex on $(0,\theta^\star)$ and by monotonicity $\psi(\theta^\star)=\psi(\theta^\star-)$ and
$\psi'(\theta^\star)=\psi'(\theta^\star-)$ are well defined as well as the right derivative at zero $\psi'(0)=\psi'(0+)=\E(Z_1)$, that we assume to be zero. We also assume that $\theta^\star>0$. In this case we can relate $\psi$ to the characteristic function $\Psi(\lambda)=-\log\E(e^{{\rm i}\lambda Z_1})$ by $\psi(\theta)=-\Psi(-{\rm i}\theta)$ for $0\le \theta < \theta^\star$.

This centered L\'evy process plays the role of Brownian Motion in the previous sections.

The generator of $Z$ applied to a function $f\in C_0^2$, the class of compactly supported functions with continuous second derivatives, gives
\[
\L f(x)=\frac12 \sigma^2 f''(x) + bf'(x) + \int_\R (f(x+y) - f(x) - yf'(x)\one{\{|y|\le 1\}} ) \Pi(dy).
\]
The adjoint of $\L$ is also well defined in $C_0^2$ and has the form
\[
 \L^*f(x)=\frac12 \sigma^2 f''(x) - bf'(x) + \int_\R (f(x-y) - f(x) + yf'(x)\one{\{|y|\le 1\}} ) \Pi(dy).
\]

Now, for $c>0$ we consider the {\em drifted process} $Z^c$ given by
\[
 Z^c_t = Z_t - ct
\]
It is immediate to see that the Laplace exponent of $Z^c$ is given by
$\psi_c(\theta)=\psi(\theta) - c \theta$ for $\theta \in [0,\theta^\star]$, that
$C_0^2$ is contained in the domain of the generator $\L_c$ of $Z^c$, and that
$\L_c f= \L f - cf'$.
Recall that the forward Kolmogorov equation for $Z$ is given by
$${d \over dt} E^x(f(Z_t))= \L f(x),$$
while the forward Kolmogorov (or Fokker-Plank) equation for the density $u$ (which exists  since $\sigma>0$) is given by
$${d \over dt} u(t,x)= \L^*u (t,\cdot)(x).$$

As in the Brownian case, we consider

\begin{itemize}
\item A branching L\'evy process (BLP) $(N_t, (\xi_t(1),\dots, \xi_t(N_t)))$ driven by $\L$.
\item A branching L\'evy process with selection of the $N$ rightmost particles ($N-$BLP), also driven by $\L$.
\item A Fleming-Viot process driven by $\L_c$ (FV).
\end{itemize}

We focus on the last two processes. For a detailed account on BLP, we refer to \cite{Ky}.
Let us just mention that the KPP equation can be
generalized in this context to:
\begin{equation}
 \label{KPP.L}
\begin{array}{l}
\displaystyle \frac{\partial v}{\partial t}=  \L v + r f(v), \quad x \in \R,\,
t>0,\\
\\
\displaystyle v(0,x)= \displaystyle v_0(x), \quad x \in \R.
\end{array}
\end{equation}

A characterization of the traveling waves as well as sufficient conditions of existence are then provided in \cite{Ky}.

 For $N-$BLP we expect (but a proof is lacking) that the empirical measure converges to a deterministic measure whose density is the solution of the generalized Durrett-Remenik equation
\begin{equation}
 \label{DR.Levy}
\begin{array}{rcl}
\mbox{{\em Find }}(\gamma,u)\,\, \mbox{{\em  such that}}& \\
\\
\displaystyle \frac{\partial u}{\partial t}(t,x) &=& \displaystyle \L^*u(t,x) + r u(t,x), \quad x>\gamma(t),\\
\displaystyle \int_{\gamma(t)}^\infty u(t,y)\, dy &=& 1,\quad u(t,x)=0, \quad x\le \gamma(t),\\
u(0,x)&=&u_0(x), \qquad  x\ge 0.
\end{array}
\end{equation}
Existence and uniqueness of solutions to this problem have to be examined, but we believe that the
proof presented in \cite{DR} can be extended to this context.

We show below the existence of traveling wave solutions for this equation under mild conditions on $\L$ based on the existence of QSDs.

Concerning FV, it is known \cite{V} that the empirical measure converges to the deterministic process given by the conditioned evolution of the process, which has a density for all times and verifies
\begin{equation}
 \label{cond.ev.Levy}
\begin{array}{rcl}
\displaystyle \frac{\partial u}{\partial t}(t,x) &=  & \displaystyle\L^* u(t,x) + c \frac{\partial u}{\partial
x}(t,x)  - u(t,x)\int_\R\L^*u(t,y,)dy \quad t>0, x>0,\\
u(t,0)&=& u(t,+\infty)=0,  \quad t>0,\\
\end{array}
\end{equation}

%

\subsection{Traveling waves vs. QSDs}

We now establish the link between traveling waves and QSD.

\begin{proposition}
\label{qsd.iff.tw}
The following statements are equivalent:
\begin{itemize}
\item
The probability measure $\nu$ with density $w$ is a QSD for $Z^c$ with eigenvalue
$-r$ ,
\item $u(t,x)=w(x-ct)$ is a traveling wave solution with speed $c$ for the free-boundary problem \eqref{DR.Levy}, with parameter $r$.
\end{itemize}
\end{proposition}

\begin{proof}
Denote $\langle f,g \rangle= \int f(x)g(x) dx$.
A QSD $\nu$ for $Z^c$ with eigenvalue
$-r$ is a solution of the equation
\[
\langle \nu  \L_c + r \nu , f  \rangle= 0, \forall f \in C_0^2.
\]
Using that $ \L_c^* f= \L^* f + c f' $ and writing that $\nu$ has density $w$, we
obtain that
\[
\L^* w + c w' + r w=0,
\]
which in turn is clearly equivalent to $w$ being a traveling wave solution with speed $c$ for \eqref{DR.Levy}.
\end{proof}

\subsection{Minimal QSD and minimal velocities}

We now study the relation between minimal traveling waves and  minimal QSD.
Using known results on QSDs for L\'evy processes, we can describe the traveling waves
for the corresponding set of equations. Since we rely on results of Kyprianou and Palmowski \cite{KP1}, we assume that $Z$
is non-lattice and consider the following two classes.

\begin{definition}
The L\'evy process $Z$ belongs to class {\rm C1} if
there exists $0<\theta_0<\theta^\star$ such that $\psi'(\theta_0)>0$ and
the process
$(Z,\P^{\theta_0})$ is in the domain of attraction of a stable law with index $1 < \alpha \le 2$, where the probability $\P^{\theta}$ is given by
\[
\left. \frac{\d\P^\theta_z}{\d\P_z}\right|_{\F_t} = e^{\theta (Z_t - z) -
\psi(\theta)t}.
\]
\end{definition}

\begin{definition}
The process $Z$ is in class {\rm C2} if $- \infty< \psi'(\theta^\star)<0$, and
 the function $x \to \Pi_{\theta^\star} \big([x,\infty) \big)$ is regularly varying at infinity with index $-\beta < -2$, where $\Pi_{\theta}(dx)=
e^{\theta x} \Pi(dx)$.
\end{definition}

Recall that we are also assuming $\psi'(0)=0$ and hence there exists a critical
$c^*$, possibly infinity, such that for $c\le c^*$, $\psi_c$ attains its negative
infimum at a point $\theta_c\le \theta^\star$ with $\psi_c'(\theta_c)=0$ and for
 $c>c^*$, the negative infimum is attained at $\theta^\star$. In this case we write
$\theta_c=\theta^\star$. Since $\psi_c(\theta)= \psi(\theta) -c \theta$, observe that $e^{\theta (Z^c_t - z) -
\psi_c(\theta)t} = e^{\theta (Z_t - z) -
\psi(\theta)t}$. Hence, if $Z$ is in class {\rm C1},
 then for $c\le c^*$, $Z^c$ is in class A in the sense of Bertoin and Doney \cite{BD,KP1}. Similarly if $Z$ is in
class {\rm C2} and $c>c^*$, then $Z^c$ is in class B in the sense of Bertoin and
Doney.

The union of class A and B represents a very large family of L\'evy
processes including for instance Brownian motion and spectrally
negative (positive) processes, as well as many others. The following theorem is proved in \cite{KP1}.
\begin{theorem}[Kyprianou and Palmowski, \cite{KP1}] \label{teo.KP} Assume $c\le
c^*$ and $Z$ is in class {\rm C1} or $c>c^*$ and $Z$ is in class {\rm C2}. Then the Yaglom limit
of $Z^c$ exists and is given by
\[
 \nu(dx)=\theta_0 \kappa_{\theta_c}(0,\theta_c) e^{-\theta_c x}V_{\theta_c}(x)\,
dx.
\]
Here $\kappa_\theta$ and $V_\theta$  are respectively  the Laplace exponent of
the ascending ladder process and the renewal function of the ladder heights
process corresponding to $(Z^c,\P^\theta)$.
\end{theorem}

The proof is based on a careful control of the asymptotics of the process as
$t \to \infty$ that in particular yields
\begin{equation}
\label{behavior.hitting.time}
 \P_x(\tau >t) \sim H(x,\theta_c)\ell(t)e^{\psi_c(\theta_c)t}.
\end{equation}
Here $H$ is a function that depends on the characteristic exponent of the
process and $\ell$ is a function regularly varying at infinity. This allows us to
prove the following

\begin{proposition}
The probability measure $\nu$ defined in Theorem \ref{teo.KP} is the minimal
QSD.
\end{proposition}

\begin{proof}
Recall that under a QSD, the hitting time of zero is exponentially
distributed. Since $\nu$ is the Yaglom limit, for any bounded function $f$
 $$\int f d \nu = \lim_{t \to \infty} \frac{\E_x(f(X_t), \tau >t)}{\P_x(\tau>t)}.$$
Applying this to $f(y)= \P_y(\tau>s)$ and using Corollary 4 in \cite{KP1}, one obtains that
$$\P_\nu(\tau>s) = \lim_{t \to \infty} \frac{\P_x(\tau >t+s)}{\P_x(\tau>t)} = \exp(\psi_c(\theta_c) s).$$
Hence the parameter of $\nu$ equals $-\psi_c(\theta_c)$.
 If  $\nu$ is not minimal, there
exists another QSD $\tilde \nu$, with parameter $\tilde r >- \psi_c(\theta_c)$.
Then $\E_{\tilde \nu}(e^{-\psi_c(\theta_c) \tau}) < \infty$ and as a consequence there exists $x >0$ such that
$\E_x(e^{-\psi_c  (\theta_c) \tau}) < \infty,$
but this contradicts \eqref{behavior.hitting.time}.

\end{proof}

%
%

Let $\C$ be the subset of $\mathbb R^2$ such that $(c,r) \in \C$ if and only if there exists a QSD $\nu$ for $Z^c$ with eigenvalue $-r$. Proposition \ref{qsd.iff.tw} states that this set coincides with the set of pairs $(c,r)$ such that there exists a traveling wave for \eqref{DR.Levy} with velocity $c$.

\begin{definition}
We say that $r$ is maximal for $c$ if $r=\max \{r' \colon (c,r') \in \C\}$. In the same way, $c$ is minimal for $r$ if $c=\min \{c' \colon (c',r) \in \C\}$.
\end{definition}

\begin{proposition}
Under the same hypotheses of Theorem \ref{teo.KP} we have that $c$ is minimal
for $r$ if and only if $r$ is maximal for $c$.
\end{proposition}

\begin{proof}
For each $c>0$, the maximal absorption rate is given by
$r(c)=-\psi_c(\theta_c)$. The proof follows by observing that the function
$c \mapsto \psi_c(\theta_c)$ is strictly decreasing and continuous.
\end{proof}

\section{Conclusions and general conjectures}
We emphasized the direct links between quasi stationary distributions QSDs for space invariant Markov processes in $\R$ and traveling waves of Durrett-Remenik equation, that are closely related to the generalized F-KPP equation. We proved also that minimal QSDs correspond to minimal velocity traveling waves. As a general fact both macroscopic and microscopic selection principles are expected to hold for QSDs and for traveling waves. They have been proved for a series of models. This suggests that the selection principles in front propagation and in QSDs are one and the same. For random walks in $\R$, the microscopic selection principle is an open problem in both cases (traveling waves and QSD).

\bigskip

\noindent {\bf Acknowledgments.}
We would like to thank UBACyT 20020090100208,
ANPCyT PICT No. 2008-0315, CONICET PIP 2010-0142 and 2009-0613 and MATHAMSUD's
project ``Stochastic structure of large interacting systems" for financial support.

\bibliographystyle{plain}
\bibliography{biblio2}

\end{document}